\documentclass[10pt]{article}
\usepackage{amsmath,amssymb,amsthm, hyperref, euscript}
\usepackage[matrix,arrow,curve]{xy}
\usepackage{graphicx}
\usepackage{tabularx}
\usepackage{float}
\usepackage{hyperref}
\usepackage{tikz}
\usetikzlibrary{matrix}

\usepackage[T1]{fontenc}

\usepackage[sc]{mathpazo}
\DeclareFontFamily{T1}{pzc}{}
\DeclareFontShape{T1}{pzc}{m}{it}{1.8 <-> pzcmi8t}{}
\DeclareMathAlphabet{\mathpzc}{T1}{pzc}{m}{it}

\title{The Unique Path Lifting for Noncommutative Covering Projections}

\theoremstyle{plain}
\newtheorem{prop}{Proposition}[section]

\newtheorem{lem}[prop]{Lemma}
\newtheorem{thm}[prop]{Theorem}

\theoremstyle{definition}
\newtheorem{defn}[prop]{Definition}
\newtheorem{empt}[prop]{}
\newtheorem{rem}[prop]{Remark}

\theoremstyle{remark}

\chardef\bslash=`\\ 

\newbox\ncintdbox \newbox\ncinttbox 
\setbox0=\hbox{$-$} \setbox2=\hbox{$\displaystyle\int$}
\setbox\ncintdbox=\hbox{\rlap{\hbox
    to \wd2{\hskip-.125em \box2\relax\hfil}}\box0\kern.1em}
\setbox0=\hbox{$\vcenter{\hrule width 4pt}$}
\setbox2=\hbox{$\textstyle\int$} \setbox\ncinttbox=\hbox{\rlap{\hbox
    to \wd2{\hskip-.175em \box2\relax\hfil}}\box0\kern.1em}

\begin{document}
\maketitle  \setlength{\parindent}{0pt}
\begin{center}
\author{}
{\textbf{Petr R. Ivankov*}\\
e-mail: * monster.ivankov@gmail.com \\
}
\end{center}

\vspace{1 in}

\begin{abstract}
\noindent

This article contains a noncommutative generalization of the topological path lifting problem. Noncommutative geometry has no paths and even points. However there are paths of *-automorphisms. It is proven that paths of *-automorphisms comply with unique path lifting.

\end{abstract}
\tableofcontents

\section{Introduction}
\begin{empt}
There is a significant problem in the algebraic topology, called the lifting problem. Let $p: \mathcal{E} \to \mathcal{B}$ and $f: \mathcal{X} \to \mathcal{B}$ continuous maps of topological spaces. The {\it lifting problem} \cite{spanier:at} for $f$ is to determine whether there is a continuous map $f': \mathcal{X} \to \mathcal{E}$ such that $f=p\circ f'$-that is, whether the dotted arrow in the diagram
\newline
\hspace*{\fill}
\begin{tikzpicture}
  \matrix (m) [matrix of math nodes,row sep=3em,column sep=4em,minimum width=2em] {
     	 & \mathcal{E} \\
     \mathcal{X} & \mathcal{B} \\};
  \path[-stealth]
    (m-2-1.east|-m-2-2) edge node [below] {$f$}(m-2-2)
    (m-1-2) edge node [right] {$p$} (m-2-2)
     (m-2-1)      edge [dashed]  node[above] {$f'$} (m-1-2);
\end{tikzpicture}
\hspace{\fill}
\newline
corresponds to a continuous map making the diagram commutative. If there is such map $f'$, then $f$ can be {\it lifted} to $\mathcal{E}$, and we call $f'$ a {\it lifting} or {\it lift} of $f$.
If $p$ is a covering projection and $\mathcal{X} = [0,1] \subset \mathbb{R}$ then $f$ can be lifted.
\end{empt}
\begin{defn}\cite{spanier:at}
A continous map $p:\mathcal{E} \to \mathcal{B}$ is said to have the {\it unique path lifting} if, given paths $\omega$ and $\omega'$ in $E$ such that $p \circ \omega = p \circ \omega'$ and $\omega(0)=\omega'(0)$, then $\omega=\omega'$.
\end{defn}
\begin{thm}\cite{spanier:at}\label{spanier_thm_un}
Let $p: \widetilde{\mathcal{X}} \to \mathcal{X}$ be a covering projection and let $f, g: \mathcal{Y} \to \widetilde{\mathcal{X}}$ be liftings of the same map (that is, $p \circ f = p \circ g$). If $\mathcal{Y}$ is connected and $f$ agrees with g for some point of $\mathcal{Y}$ then $f=g$.
\end{thm}
\begin{rem}
From theorem \ref{spanier_thm_un} it follows that a covering projection has unique path lifting.
\end{rem}
\paragraph*{} We would like generalize above facts. Because noncommutative geometry has no points it has no a direct generalization of paths, but there is an implicit generalization. Let $\widetilde{\mathcal{X}} \to \mathcal{X}$ be a covering projection. The following diagram reflects the path lifting problem.
\newline
\hspace*{\fill}
\begin{tikzpicture}
  \matrix (m) [matrix of math nodes,row sep=3em,column sep=4em,minimum width=2em] {
     	 & \widetilde{\mathcal{X}} \\
     I = [0,1]  & \mathcal{X} \\};
  \path[-stealth]
    (m-2-1.east|-m-2-2) edge node [below] {$f$}(m-2-2)
    (m-1-2) edge node [right] {$p$} (m-2-2)
     (m-2-1)      edge [dashed]  node[above] {$f'$} (m-1-2);
\end{tikzpicture}
\hspace{\fill}
\newline
However above diagram can be replaced with an equivalent diagram
\newline
\hspace*{\fill}
 \begin{tikzpicture}
  \matrix (m) [matrix of math nodes,row sep=3em,column sep=4em,minimum width=2em] {
     	 & \mathrm{Homeo}(\widetilde{\mathcal{X}}) \\
     I = [0,1]  & \mathrm{Homeo}(\mathcal{X}) \\};
  \path[-stealth]
    (m-2-1.east|-m-2-2) edge node [below] {$f$}(m-2-2)
    (m-1-2) edge node [right] {$p$} (m-2-2)
     (m-2-1)      edge [dashed]  node[above] {$f'$} (m-1-2);
\end{tikzpicture}
\hspace{\fill}
 \newline
where $\mathrm{Homeo}$ means the group of homeomorphisms with {\it compact-open} topology (See \cite{spanier:at}). Noncommutative generalization of a locally compact space is a $C^*$-algebra \cite{arveson:c_alg_invt}. Similarly the generalization of $\mathrm{Homeo}(\mathcal{X})$ is the group  $\mathrm{Aut}(A)$ of *-automorphisms carries (at least) two different topologies making it into a topological group \cite{thomsem:ho_type_uhf}. The most important is {\it the topology of pointwise norm-convergence} based on the open sets
\begin{equation*}
\left\{\alpha \in \mathrm{Aut}(A) \ | \ \|\alpha(a)-a\| < 1 \right\}, \ a \in A.
\end{equation*}
\paragraph*{} The other topology is the {\it uniform norm-topology} based on the open sets
\begin{equation*}
\left\{\alpha \in \mathrm{Aut}(A) \ | \ \sup_{a \neq 0}\ \|a\|^{-1} \|\alpha(a)-a\| < \varepsilon \right\}, \ \varepsilon > 0
\end{equation*}
which corresponds to following "norm"
\begin{equation}\label{uniform_norm_topology_formula}
\|\alpha\| = \sup_{a \neq 0}\ \|a\|^{-1} \|\alpha(a)-a\| .
\end{equation}
Above formula does not really means a norm because $\mathrm{Aut}\left(A\right)$ is not a vector space. Henceforth the uniform norm-topology will be considered only.
Following diagram presents the noncommutative generalization of path lifting.
\newline
 \hspace*{\fill}
 \begin{tikzpicture}
  \matrix (m) [matrix of math nodes,row sep=3em,column sep=4em,minimum width=2em] {
     	 & \mathrm{Aut}\left(\widetilde{A}\right) \\
     I = [0,1]  & \mathrm{Aut}\left(A\right) \\};
  \path[-stealth]
    (m-2-1.east|-m-2-2) edge node [below] {$f$}(m-2-2)
    (m-1-2) edge node [right] {$p$} (m-2-2)
     (m-2-1)      edge [dashed]  node[above] {$f'$} (m-1-2);
\end{tikzpicture}
\hspace{\fill}
\newline

 A generalization of covering projections is described in my articles \cite{ivankov:cov_pr_nc_torus_fin,ivankov:infinite_cov_pr}. Following table contains necessary ingredients and their noncommutative analogues.
 \newline
 \newline
 \begin{tabular}{|c|c|}
 \hline
 General topology & Noncommutative geometry\\
 \hline
Locally compact Hausdorff space \cite{bourbaki_sp:gt} & $C^*$-algebra \cite{arveson:c_alg_invt}\\
A group of homeomorphism $\mathrm{Homeo}(X)$ \cite{spanier:at} & A group of *-automorphisms $\mathrm{Aut}\left(A\right)$ \cite{thomsem:ho_type_uhf} \\
The topology on $\mathrm{Homeo}\left(\mathcal{X}\right)$ \cite{spanier:at}    & The topology on $\mathrm{Aut}\left(A\right)$ \cite{thomsem:ho_type_uhf}  \\
Covering projection \cite{spanier:at}  & Noncommutative covering projection \cite{ivankov:cov_pr_nc_torus_fin,ivankov:infinite_cov_pr} \\

 \hline
 \end{tabular}
 \newline
 \newline
 \break
 Composition of these ingredients supplies a noncommutative generalization of the following topological theorem.
 \begin{thm}\label{comm_thm}
 Let $p: \widetilde{\mathcal{X}} \to \mathcal{X}$ be a covering projection. Any path $\omega: I \to \mathrm{Homeo}(\mathcal{X})$ such that $\omega(0)=\mathrm{Id}_{\mathcal{X}}$ can be uniquely lifted to the path $\widetilde{\omega}: I \to \mathrm{Homeo}(\widetilde{\mathcal{X}})$ such that $\widetilde{\omega}(0)=\mathrm{Id}_{\widetilde{\mathcal{X}}}$ , i.e. $p(\widetilde{\omega}(t)(x)) = \omega(t)(p(x))$, $\forall t \in [0,1], \ \forall x \in \widetilde{\mathcal{X}}$.
 \end{thm}
 \begin{proof}
Follows from theorem \ref{spanier_thm_un}.
 \end{proof}

\paragraph{} Following notation is used in this article.
\newline
\newline
 \begin{tabular}{|c|c|}
 \hline
 Symbol & Meaning\\
 \hline
 $\mathrm{Aut}(A)$ & Group * - automorphisms of $C^*$-algebra $A$\\
 $\mathbb{C}$ (resp. $\mathbb{R}$)  & Field of complex (resp. real) numbers \\
$C(\mathcal{X})$ & $C^*$ - algebra of continuous complex valued
  functions on compact\\ & topological space $\mathcal{X}$\\

$C_0(\mathcal{X})$ & $C^*$ - algebra of continuous complex valued
  functions on locally compact \\ & topological space which tends to 0 at infinity $\mathcal{X}$\\

$M(A)$  & Multiplier algebra of $C^*$-algebra $A$\\
 $\mathbb{N}$ & The set of natural numbers\\

 \hline
 \end{tabular}

  \section{Noncommutative covering projections}
  \paragraph*{} In this section we recall some notions from the article \cite{ivankov:infinite_cov_pr} devoted to noncommutative covering projections.
 \subsection{Hermitian modules and functors}
\paragraph*{} Any continuous map from a compact space $\mathcal{X}$ to a compact space $\mathcal{Y}$ corresponds to the *-homomorphism $C(\mathcal{Y})\to C(\mathcal{X})$. It is not always true when $\mathcal{X}$ is not compact, a continuous map is rather a correspondence \cite{kakariadis:corr} between $C_0(\mathcal{Y})$ and $C_0(\mathcal{X})$.
 In this section we consider an algebraic generalization of continuous maps. Following text is in fact a citation of \cite{rieffel_morita}.

  \begin{defn}
 \cite{rieffel_morita} Let $B$ be a $C^*$-algebra. By a (left) {\it Hermitian $B$-module} we will mean the Hilbert space $H$ of a non-degenerate *-representation $A \rightarrow B(H)$. Denote by $\mathbf{Herm}(B)$ the category of Hermitian $B$-modules.
  \end{defn}
 \begin{empt}
 \cite{rieffel_morita} Let $A$, $B$ be $C^*$-algebras. In this section we will study some general methods for construction of functors from  $\mathbf{Herm}(B)$ to  $\mathbf{Herm}(A)$.
 \end{empt}
 \begin{defn} \cite{rieffel_morita}
 Let $B$ be a $C^*$-algebra. By (right) {\it pre-$B$-rigged space} we mean a vector space, $X$, over complex numbers on which $B$ acts by means of linear transformations in such a way that $X$ is a right $B$-module (in algebraic sense), and on which there is defined a $B$-valued sesquilinear form $\langle,\rangle_X$ conjugate linear in the first variable, such that
 \begin{enumerate}
 \item $\langle x, x \rangle_B \ge 0$
 \item $\left(\langle x, y \rangle_X\right)^* = \langle y, x \rangle_X$
 \item $\langle x, yb \rangle_B = \langle x, y \rangle_Xb$
 \end{enumerate}
 \end{defn}
 \begin{empt}
 It is easily seen that if we factor a pre-$B$-rigged space by subspace of the elements $x$ for which $\langle x, x \rangle_B = 0$, the quotient becomes in a natural way a pre-$B$-rigged space having the additional property that inner product is definite, i.e. $\langle x, x \rangle_X > 0$ for any non-zero $x\in X$. On a pre-$B$-rigged space with definite inner product we can define a norm $\|\|$ by setting
 \begin{equation}\label{rigged_norm_eqn}
 \|x\|=\|\langle x, x \rangle_X\|^{1/2}, \
 \end{equation}
 From now on we will always view a  pre-$B$-rigged space  with definite inner product as being equipped with this norm. The completion of $X$ with this norm is easily seen to become again a pre-$B$-rigged space.
 \end{empt}
 \begin{defn}
 \cite{rieffel_morita} Let $B$ be a $C^*$-algebra. By a {\it $B$-rigged space} or {\it Hilbert $B$-module} we will mean a pre-$B$-rigged space, $X$, satisfying the following conditions:
 \begin{enumerate}
 \item If $\langle x, x \rangle_X\ = 0$ then $x = 0$, for all $x \in X$,
 \item $X$ is complete for the norm defined in (\ref{rigged_norm_eqn}).
 \end{enumerate}
 \end{defn}
 \begin{rem}
 In many publications the "Hilbert $B$-module" term is used instead "rigged $B$-module".
 \end{rem}
 \begin{empt}
 Viewing a $B$-rigged space as a generalization of an ordinary Hilbert space, we can define what we mean by bounded operators on a $B$-rigged space.
 \end{empt}
 \begin{defn}\cite{rieffel_morita}
 Let $X$ be a $B$-rigged space. By a {\it bounded operator} on $X$ we mean a linear operator, $T$, from $X$ to itself which satisfies following conditions:
 \begin{enumerate}
 \item for some constant $k_T$ we have
 \begin{equation}\nonumber
 \langle Tx, Tx \rangle_X \le k_T \langle x, x \rangle_X, \ \forall x\in X,
 \end{equation}
 or, equivalently $T$ is continuous with respect to the norm of $X$.
 \item there is a continuous linear operator, $T^*$, on $X$ such that
 \begin{equation}\nonumber
 \langle Tx, y \rangle_X = \langle x, T^*y \rangle_X, \ \forall x, y\in X.
 \end{equation}
 \end{enumerate}
 It is easily seen that any bounded operator on a $B$-rigged space will automatically commute with the action of $B$ on $X$ (because it has an adjoint). We will denote by $\mathcal{L}(X)$ (or $\mathcal{L}_B(X)$ there is a chance of confusion) the set of all bounded operators on $X$. Then it is easily verified than with the operator norm $\mathcal{L}(X)$ is a $C^*$-algebra.
 \end{defn}
 \begin{defn}\cite{pedersen:ca_aut} If $X$ is a $B$-rigged module then
 denote by $\theta_{\xi, \zeta} \in \mathcal{L}_B(X)$   such that
 \begin{equation}\nonumber
 \theta_{\xi, \zeta} (\eta) = \zeta \langle\xi, \eta \rangle_X , \ (\xi, \eta, \zeta \in X)
 \end{equation}
 Norm closure of  a generated by such endomorphisms ideal is said to be the {\it algebra of compact operators} which we denote by $\mathcal{K}(X)$. The $\mathcal{K}(X)$ is an ideal of  $\mathcal{L}_B(X)$. Also we shall use following notation $\xi\rangle \langle \zeta \stackrel{\text{def}}{=} \theta_{\xi, \zeta}$.
 \end{defn}

 \begin{defn}\cite{rieffel_morita}\label{corr_defn}
 Let $A$ and $B$ be $C^*$-algebras. By a {\it Hermitian $B$-rigged $A$-module} we mean a $B$-rigged space, which is a left $A$-module by means of *-homomorphism of $A$ into $\mathcal{L}_B(X)$.
 \end{defn}
 \begin{rem}
 Hermitian $B$-rigged $A$-modules are also named  as {\it $B$-$A$-correspondences} (See, for example \cite{kakariadis:corr}).
 \end{rem}
 \begin{empt}\label{herm_functor_defn}
 Let $X$ be a Hermitian $B$-rigged $A$-module. If $V\in \mathbf{Herm}(B)$ then we can form the algebraic tensor product $X \otimes_{B_{\mathrm{alg}}} V$, and equip it with an ordinary pre-inner-product which is defined on elementary tensors by
 \begin{equation}\nonumber
 \langle x \otimes v, x' \otimes v' \rangle = \langle \langle x',x \rangle_B v, v' \rangle_V.
 \end{equation}
 Completing the quotient $X \otimes_{B_{\mathrm{alg}}} V$ by subspace of vectors of length zero, we obtain an ordinary Hilbert space, on which $A$ acts (by $a(x \otimes v)=ax\otimes v$) to give a  *-representation of $A$. We will denote the corresponding Hermitian module by $X \otimes_{B} V$. The above construction defines a functor $X \otimes_{B} -: \mathbf{Herm}(B)\to \mathbf{Herm}(A)$ if for $V,W \in \mathbf{Herm}(B)$ and $f\in \mathrm{Hom}_B(V,W)$ we define $f\otimes X \in \mathrm{Hom}_A(V\otimes X, W\otimes X)$ on elementary tensors by $(f \otimes X)(x \otimes v)=x \otimes f(v)$.	We can define action of $B$ on $V\otimes X$ which is defined on elementary tensors by
 \begin{equation}\nonumber
 b(x \otimes v)= (x \otimes bv) = x b \otimes v.
 \end{equation}
 \end{empt}

\subsection{Galois rigged modules}
\begin{defn}\label{herm_a_g_defn}\cite{ivankov:infinite_cov_pr} Let $A$ be a $C^*$-algebra, $G$ is a finite or countable group which acts on $A$. We say that  $H \in \mathbf{Herm}(A)$ is a {\it $A$-$G$ Hermitian module} if
\begin{enumerate}
\item Group $G$ acts on $H$ by unitary $A$-linear isomorphisms,
\item There is a subspace $H^G \subset H$ such that
\begin{equation}\label{g_act}
H = \bigoplus_{g\in G}gH^G.
\end{equation}
\end{enumerate}
Let $H$, $K$ be  $A$-$G$ Hermitian modules, a morphism $\phi: H\to K$ is said to be a $A$-$G$-morphism if $\phi(gx)=g\phi(x)$ for any $g \in G$. Denote by $\mathbf{Herm}(A)^G$ a category of  $A$-$G$ Hermitian modules and $A$-$G$-morphisms.
\end{defn}

\begin{defn}\cite{ivankov:infinite_cov_pr}
Let $H$ be $A$-$G$ Hermitian module, $B\subset M(A)$ is sub-$C^*$-algebra such that $(ga)b = g(ab)$,  $b(ga) = g(ba)$, for any $a\in A$, $b \in B$, $g \in G$.
There is a functor $(-)^G: \mathbf{Herm}(A)^G \to\mathbf{Herm}(B)$ defined by following way
\begin{equation}
H \mapsto H^G.
\end{equation}
This functor is said to be the {\it invariant functor}.
\end{defn}

\begin{defn}\cite{ivankov:infinite_cov_pr}
Let $_AX_B$ be a Hermitian $B$-rigged $A$-module, $G$ is finite or countable group such that
\begin{itemize}
\item $G$ acts on $A$ and $X$,
\item Action of $G$ is equivariant, i.e $g (a\xi) = (ga) (g\xi)$ , and $B$ invariant, i.e $g(\xi b)=(g\xi)b$ for any $\xi \in X$, $b \in B$, $a\in A$, $g \in G$,
\item Inner-product  of $G$ is equivariant, i.e $\langle g\xi, g \zeta\rangle_X = \langle\xi, \zeta\rangle_X$ for any $\xi, \zeta \in X$,  $g \in G$.
\end{itemize}
Then we say that  $_AX_B$ is a {\it $G$-equivariant $B$-rigged $A$-module}.
\end{defn}
\begin{empt}
Let $_AX_B$ be a  $G$-equivariant $B$-rigged $A$-module. Then for any  $H\in \mathbf{Herm}(B)$ there is an action of $G$  on $X\otimes_B H$ such that
\begin{equation}
g \left(x \otimes \xi\right) = \left(x \otimes g\xi\right).
\end{equation}

\end{empt}

\begin{defn}\label{inf_galois_defn}\cite{ivankov:infinite_cov_pr}
Let $_AX_B$ be a $G$-equivariant $B$-rigged $A$-module. We say that  $_AX_B$ is {\it $G$-Galois $B$-rigged $A$-module} if it satisfies following conditions:
\begin{enumerate}
\item  $X \otimes_B H$ is a $A$-$G$ Hermitian module, for any $H \in \mathbf{Herm}(B)$,
\item A pair   $\left(X \otimes_B -, \left(-\right)^G\right)$ such that
\begin{equation}\nonumber
X \otimes_B -: \mathbf{Herm}(B) \to \mathbf{Herm}(A)^G,
\end{equation}
\begin{equation}\nonumber
(-)^G: \mathbf{Herm}(A)^G \to \mathbf{Herm}(B).
\end{equation}

is a pair of inverse equivalence.

\end{enumerate}

\end{defn}

\begin{thm}\label{nc_inf_cov_thm}\cite{ivankov:infinite_cov_pr}
Let $A$ and $\widetilde{A}$ be $C^*$-algebras,  $_{\widetilde{A}}X_A$ be a $G$-equivariant $A$-rigged $\widetilde{A}$-module. Let $I$ be a finite or countable set of indices,  $\{e_i\}_{i\in I} \subset A$, $\{\xi_i\}_{i\in I} \subset \ _{\widetilde{A}}X_A$ such that
\begin{enumerate}
\item
\begin{equation}\label{1_mb}
1_{M(A)} =  \sum_{i\in I}^{}e^*_ie_i,
\end{equation}
\item
\begin{equation}\label{1_mkx}
1_{M(\mathcal{K}(X))} = \sum_{g\in G}^{} \sum_{i \in I}^{}g\xi_i\rangle \langle g\xi_i ,
\end{equation}
\item
\begin{equation}\label{ee_xx}
\langle \xi_i, \xi_i \rangle_X = e_i^*e_i,
\end{equation}
\item
\begin{equation}\label{g_ort}
\langle g\xi_i, \xi_i\rangle_X=0, \ \text{for any nontrivial} \ g \in G.
\end{equation}
\end{enumerate}
Then $_{\widetilde{A}}X_A$ is a $G$-Galois $A$-rigged $\widetilde{A}$-module.

\end{thm}
\begin{defn}\label{sub_def}\cite{ivankov:infinite_cov_pr}
Consider a situation from the theorem \ref{nc_inf_cov_thm}. Norm completion of the generated by operators
\begin{equation*}
g\xi_i^* \rangle \langle g \xi_i a; \ g \in G, \ i \in I, \ a \in M(A)
\end{equation*}
algebra is said to be the {\it subordinated to $\{\xi_i\}_{i \in I}$ algebra}. If $\widetilde{A}$ is the subordinated to $\{\xi_i\}_{i \in I}$ then
\begin{enumerate}
\item $G$ acts on $\widetilde{A}$ by following way
\begin{equation*}
g \left( \ g'\xi_i^* \rangle \langle g' \xi_i a \right) =  gg'\xi_i^* \rangle \langle gg' \xi_i a; \ a,a \in M(A).
\end{equation*}
\item $X$ is a left $A$ module, moreover $_{\widetilde{A}}X_A$ is a  $G$-Galois $A$-rigged $\widetilde{A}$-module.
\item There is a natural $G$-equivariant *-homomorphism $\varphi: A \to M\left(\widetilde{A}\right)$, $\varphi$ is equivariant, i.e.
\begin{equation}
 \varphi(a)(g\widetilde{a})= g \varphi(a)(\widetilde{a}); \ a \in A, \ \widetilde{a}\in \widetilde{A}.
\end{equation}
\end{enumerate}
A quadruple $\left(A, \widetilde{A}, _{\widetilde{A}}X_A, G\right)$ is said to be a {\it Galois quadruple}.
\end{defn}
\begin{rem}
It is shown \cite{ivankov:infinite_cov_pr} that, if $A$ is a commutative $C^*$-algebra then the Galois quadruple $\left(A, \widetilde{A}, _{\widetilde{A}}X_A, G\right)$ corresponds to a topological covering projection. So a Galois quadruple is a generalization of a topological covering projection.
\end{rem}
\begin{rem}
Henceforth subordinated algebras only are regarded as noncommutative generalizations of covering projections.
\end{rem}

\begin{lem}\label{aut_norm_lem}
Let $\left(A, \widetilde{A}, _{\widetilde{A}}X_A, G\right)$ be a Galois quadruple, and let $g \in G$ be a nontrivial element. Then $\|g\| \ge 1$ where $\|g\|$ is given by \eqref{uniform_norm_topology_formula}.
\end{lem}
\begin{proof}
Let  $x =  \ g'\xi_i^* \rangle \langle g' \xi_i a$, ($a\in M(A)$) then from \eqref{g_ort} it follows that
\begin{equation*}
x (gx^*) = x^* (gx)= (gx)x^*=(gx^*)x =0.
\end{equation*}
 From
\begin{equation*}
\|x - gx\|^2 = \|(x - gx)^* (x - gx)\| = \|xx^* + (gx)(gx^*)\| \ge \|x\|^2.
\end{equation*}
it follows that $\|g\| \ge 1$.
\end{proof}
\begin{lem}\label{shift_lem}
Let $\left(A, \widetilde{A}, _{\widetilde{A}}X_A, G\right)$ be a Galois quadruple, and let $\alpha \in \mathrm{Aut}\left(\widetilde{A}\right)$ be such that $\alpha$ is a right $A$-module isomorphism, i.e.
\begin{equation*}
\alpha(\widetilde{a})a = \alpha(\widetilde{a}a); \ \forall \widetilde{a} \in \widetilde{A}, \ \forall a \in A.
\end{equation*}
Then $\alpha \in G$.
\end{lem}
\begin{proof}
Let $\mathfrak{I}_i = e^*_ie_iM(A)$, ($i \in I$) a right principal ideal in $A$. Then $\widetilde{\mathfrak{I}}_i=\mathfrak{I}_iM(A)=e^*_ie_i\widetilde{A}$ is a right $A$-module which is a Hilbert direct sum
\begin{equation}\label{ideal_dir_sum}
\widetilde{\mathfrak{I}}_{i} = \bigoplus_{g \in G} \widetilde{\mathfrak{I}}_{ig}
\end{equation}
where $\widetilde{\mathfrak{I}}_{ig} = g\xi_i^* \rangle \langle g \xi_i M(A)$.
From conditions of theorem \ref{nc_inf_cov_thm} if follows that there is a natural right $A$-module isomorphism $\mathfrak{J}_i \approx \widetilde{\mathfrak{J}}_{ig}$ for any ($g\in G$). The $A$- module isomorphism $\alpha$ transposes summands of \eqref{ideal_dir_sum}. Since $G$ transitively acts on itself there is $g \in G$ such that $g$ corresponds the transposition of direct sum members. Then action of $\alpha$ is uniquely defined as
\begin{equation}\label{act_form}
\alpha(g'\xi_i^* \rangle \langle g' \xi_i a) = gg'\xi_i^* \rangle \langle gg' \xi_i; \ a \in M(A).
\end{equation}
From \eqref{act_form} it follows that $\alpha = g$.
\end{proof}

\section{The noncommutative path lifting}

\begin{defn}
Let $(A, \widetilde{A}, _{\widetilde{A}}X_A, G)$ be a Galois quadruple, and $\omega: I \to \mathrm{Aut}(A)$ is a continuous map with respect to the uniform norm-topology. A continuous  with respect to uniform norm-topology map $\widetilde{\omega}: I \to \mathrm{Aut}(\widetilde{A})$ such that
\begin{equation*}
 \widetilde{\omega}(t)(\widetilde{a}a)=\widetilde{\omega}(t)(\widetilde{a})\omega(t)(a); \ \forall \ t \in I, \ a \in A, \ \widetilde{a} \in \widetilde{A}
\end{equation*}
 is said to be a {\it lift} or {\it lifting} of $\omega$.
\end{defn}
\paragraph*{} Following lemma states uniqueness of a path lift.
\begin{lem}
Let $(A, \widetilde{A}, _{\widetilde{A}}X_A, G)$ be a Galois quadruple, and $\omega: I \to \mathrm{Aut}(A)$ is a continuous map with respect to uniform norm-topology, and let $\widetilde{\omega}_1$, $\widetilde{\omega}_2$ be lifts of  $\omega$ such that $\widetilde{\omega}_1(0)=\widetilde{\omega}_2(0)$. Then $\widetilde{\omega}_1=\widetilde{\omega}_2$.
\end{lem}
\begin{proof}
Suppose that $U = \{x \in I \ | \  \widetilde{\omega}_1(t)=\widetilde{\omega}_2(t), 0\le t \le x \}$.
 Then $U = [0,T]$. If $T < 1$ then $\widetilde{\omega}_1\ne\widetilde{\omega}_2$. Since $\widetilde{\omega}_1$, $\widetilde{\omega}_2$ are continuous maps, there is $\varepsilon > 0$ such that
\begin{equation}\label{neq_1}
\|\left(\widetilde{\omega}_i(T)\right)^{-1}\widetilde{\omega}_i(T+t)\| < 1/4; \ 0 \le t < \varepsilon,  \ i =1,2.
\end{equation}
Let $t < \varepsilon$ be such that $\widetilde{\omega}_1(T + t) \ne \widetilde{\omega}_2(T + t)$. From lemma \ref{shift_lem} it follows that there is nontrivial element $g\in G$ such that $\widetilde{\omega}_1(T + t) = g \widetilde{\omega}_2(T + t)$. From lemma \ref{aut_norm_lem} it follows that
 \begin{equation}\label{neq_2}
\|g\| = \|(\widetilde{\omega}_1(T + t))^{-1}\widetilde{\omega}_2(T + t)\| \ge 1.
 \end{equation}
 From $\widetilde{\omega}_1(T) = \widetilde{\omega}_2(T)$ it follows a contradiction between \eqref{neq_1} and \eqref{neq_2}. This contradiction proves this lemma.
 \end{proof}
 \paragraph*{}
 Now we shall prove the existence of path lifting. Let $(A, \widetilde{A}, _{\widetilde{A}}X_A, G)$ be a Galois quadruple. We suppose that right ideals and right $A$-modules given by
 \begin{equation}\label{irr_id}
\mathfrak{I}_i=e^*_ie_iM(A) \subset A,
 \end{equation}
  \begin{equation}\label{irr_mod}
\widetilde{\mathfrak{I}}_{ig} = g\xi_i^* \rangle \langle g \xi_i M(A) \subset\widetilde{A}.
  \end{equation}
 are irreducible \cite{douns:mod_ring}.
 If not we can decompose modules to direct sums of irreducible ones.
 Following lemma states the existence of path lift.
 \begin{lem}
 Let $(A, \widetilde{A}, _{\widetilde{A}}X_A, G)$ be a Galois quadruple, and $\omega: I \to \mathrm{Aut}(A)$ is a continuous map with respect to uniform norm-topology. There exist a path lift  $\widetilde{\omega}: I \to \mathrm{Aut}\left(\widetilde{A}\right)$ of $\omega$.
 \end{lem}
\begin{proof}
Let $\varepsilon > 0$ be such that, if $t < \varepsilon$ then $\|\omega(0)^{-1}\omega(t)\| < 1/ 4$. Let $i \in I$ be an index defined in the theorem \ref{nc_inf_cov_thm}. Denote by $e_i(t)=\omega(t)(e_i)$. If $t < \varepsilon$ then $\|e_i - e_i(t)\| < 1/4 \|e_i\|$ for any $i \in I$. Let $\mathfrak{I}_i$, (resp. $\widetilde{\mathfrak{I}}_{ig}$) be given by \eqref{irr_id} (resp. \eqref{irr_mod}) From \eqref{ideal_dir_sum} it follows that
\begin{equation}\label{ideal_dir_sum1}
e^*_ie_i \widetilde{A} = \mathfrak{J}_i \widetilde{A} = \bigoplus_{g \in G} \widetilde{\mathfrak{I}}_{ig}
\end{equation}
 where $\bigoplus$ means the Hilbert direct sum.
 The ideal
 \begin{equation*}
 \mathfrak{I}_i(t)=e^*_i(t)e_i(t)M(A)
 \end{equation*}
 is irreducible for any $t$, because $\omega(t)$ is a *-isomorphism.
 So
\begin{equation}
\mathfrak{I}_i(t) \widetilde{A} = e^*_i(t)e_i(t) \widetilde{A} = \bigoplus_{j \in J} \widetilde{\mathfrak{I}}_{ij}(t).
\end{equation}
where $|J| = |G|$ and all right $A$-modules $\widetilde{\mathfrak{I}}_{ij}(t)$ are irreducible and $\widetilde{\mathfrak{I}}_{ij}(t)\approx \mathfrak{I}_i(t)$, $\forall j \in J$. For any $j \in J$ there is the unique $\xi_{ij}(t)\in  \widetilde{\mathfrak{I}}_{ij}(t)$ such that $\xi_{ij}(t)$ corresponds to $e_i(t)$ by isomorphism $\widetilde{\mathfrak{I}}_{ij}(t)\approx \mathfrak{I}_i(t)$. Let $J_i(t) = \{\xi_{ij}(t)\}_{j \in J}$.  For any $0 \le t \le \varepsilon$ denote by $\xi_i(t)\in J_i(t)$ the unique element such that
\begin{equation*}
\|\xi_i - \xi_i(t)\| < 1/4 \|\xi_i\|.
\end{equation*}
Lemma \ref{aut_norm_lem} guaranties uniqueness of $\xi_i(t)$ and a map $t \mapsto \xi_i(t)$ is continuous in the norm topology.  Define $\widetilde{\omega}: [0, \varepsilon] \to \mathrm{Aut}(\widetilde{A})$ such that
\begin{equation*}
\widetilde{\omega}(t)( g\xi_i^* \rangle \langle g \xi_i a) = g\xi_i^*(t) \rangle \langle g \xi_i(t) \omega(t)(a); \ \forall \ i \in I, \ g \in G, \ a \in M(A).
\end{equation*}
Similarly there is $\varepsilon_1> 0$ such that if $0 \le t \le \varepsilon_1$ then $\|\omega(\varepsilon)^{-1}\omega(\varepsilon + t)\| < 1/ 4$ and a domain of $\widetilde{\omega}$ can be extended to the interval $[0, \varepsilon + \varepsilon_1]$. Thus the path $\widetilde{\omega}$  can be extended to the whole interval $[0,1]$ because $[0,1]$ is a compact set.
 \end{proof}
 \section{Application}
 \paragraph*{} The unique path lifting property is used in my article \cite{ivankov:nc_wilson_lines} devoted to the noncommutative generalization of Wilson lines.


\end{document}